\documentclass{amsart}

\usepackage{amsmath,amssymb,amsthm,framed,graphicx,latexsym,youngtab,epstopdf,booktabs}
\usepackage{color,fourier}

\usepackage{ array }
\usepackage{ enumitem }
\usepackage{ url }

\theoremstyle{definition}

\newtheorem{definition}{Definition}[section]

\newtheorem{example}[definition]{Example}

\theoremstyle{plain}

\newtheorem{lemma}[definition]{Lemma}
\newtheorem{proposition}[definition]{Proposition}
\newtheorem{theorem}[definition]{Theorem}

\allowdisplaybreaks

\def\calA{\mathcal{A}}
\def\calB{\mathcal{B}}

\def\calD{\mathcal{D}}

\def\calO{\mathcal{O}}
\def\calP{\mathcal{P}}
\def\calQ{\mathcal{Q}}
\def\calR{\mathcal{R}}
\def\calS{\mathcal{S}}
\def\calT{\mathcal{T}}

\def\calX{\mathcal{X}}
\def\calY{\mathcal{Y}}

\DeclareMathOperator{\Com}{\textsl{Com}}
\DeclareMathOperator{\Lie}{\textsl{Lie}}
\DeclareMathOperator{\Ass}{\textsl{Ass}}
\DeclareMathOperator{\Poisson}{\textsl{Poisson}}
\DeclareMathOperator{\NLie}{\textsl{NLie}}

\DeclareMathAlphabet{\mathbbold}{U}{bbold}{m}{n}

\def\k{\mathbbold{k}}

\begin{document}


\title{Distributive laws between the operads Lie and Com}

\author{Murray Bremner}

\address{Department of Mathematics and Statistics, University of Saskatchewan, 
Room 142 McLean Hall, 106 Wiggins Road, Saskatoon, Saskatchewan, Canada, S7N 5E6}

\email{bremner@math.usask.ca}

\author{Vladimir Dotsenko}

\address{Institut de Recherche Math\'ematique Avanc\'ee, UMR 7501, Universit\'e de Strasbourg et CNRS, 7 rue Ren\'e-Descartes, 67000 Strasbourg, France}
\email{vdotsenko@unistra.fr}





\thanks{The research of the first author was supported by a Discovery Grant from NSERC, 
the Natural Sciences and Engineering Research Council of Canada.}

\begin{abstract}
Using methods of computer algebra, especially Gr\"obner bases for submodules of free modules over polynomial rings, we solve a classification problem in theory of algebraic operads: we show that the only nontrivial (possibly inhomogeneous) distributive law between the operad of Lie algebras and the operad of commutative associative algebras is given by the Livernet--Loday formula deforming the Poisson operad into the associative operad.
\end{abstract}

\maketitle


\section{Introduction}

The most classical examples of algebraic operads, or the ``three graces of operad theory'' (an expression coined by Jean-Louis Loday), are the operad $\Lie$ controlling Lie algebras, the operad $\Com$ controlling associative commutative algebras, and the operad $\Ass$ controlling associative algebras that are not necessarily commutative. The relationship between these operads can be described informally by an exact sequence
 \[
0\to\Lie\hookrightarrow\Ass\twoheadrightarrow\Com\to 0 ,
 \]
meaning that the suboperad of $\Ass$ generated by the operation $[a_1,a_2]=a_1\cdot a_2-a_2\cdot a_1$ is isomorphic to $\Lie$, and the quotient by the ideal generated by $\Lie$ is isomorphic to $\Com$. The same is true if the operad $\Ass$ is replaced by the operad $\Poisson$ of Poisson algebras.

In fact, the underlying $\mathbb{S}$-module for both the operad $\Ass$ and the operad $\Poisson$ is obtained from those for the operads $\Com$ and $\Lie$ by the composite product: on the level of $\mathbb{S}$-modules, we have
 \[
\Ass\cong\Poisson\cong\Com\circ\Lie .
 \]
A well known observation due to Livernet and Loday (see e.g. \cite{MR2006}) is that over a field of zero characteristic there exists a one-parametric family of operads $\calO_q$ for which $\calO_0\cong\Poisson$ and $\calO_1\cong\Ass$; moreover, if all elements of the ground field are perfect squares, one has $\calO_q\cong\Ass$ for $q\ne 0$. 
In the case of associative algebras, situations like that are often studied in the context of the ``Koszul deformation principle'' of Drinfeld \cite{PP}. A full analogue of this result for operads is not available, but the universe of inhomogeneous distributive laws provides examples where a version of the Koszul deformation principle holds. For associative algebras, many related interesting examples for the Koszul deformation principle in the case of associative algebras are (noncommutative) algebras that deform polynomial algebras in some sense, e.g. admit a basis of standard monomials $x_1^{a_1}x_2^{a_2}\cdots x_n^{a_n}$. Among such relations, an important particular case is the so called PBW case, where representing elements of an algebra as a linear combinations of standard monomials can be done by term rewriting. One natural generalisation to the case of operads is obtained by looking at defining relations that produce operads whose underlying analytic functors are compositions of the underlying analytic functors of the given operads, via appropriate rewriting rules that we call inhomogeneous distributive laws; those generalise the notion of a distributive law between two operads studied by Markl~\cite{Markl1996}, a particular case of distributive laws between monads going back to the work of Beck \cite{Beck1969}. Our inhomogeneous distributive laws, while still resembling those of Beck, do not fit into that framework, as we allow the relations of one of the operads to be deformed; what is preserved is the free right module structure over the other operad, fitting into the general framework for Poincar\'e--Birkhoff--Witt theorems of the second author and Tamaroff \cite{DT2018}.

In this paper, we use computational commutative algebra to classify all inhomogeneous distributive laws between $\Com$ and $\Lie$, thus answering a question which the second author was asked in private communications with Vladimir Hinich and Jean-Louis Loday. Our main result is that the family $\calO_q$ mentioned above is the only instance of an inhomogeneous distributive law between the Lie operad and the commutative operad (except for one trivial distributive law that always exists). This generalises the classification of homogeneous distributive laws obtained by the first author and Markl \cite{BM2018}. Pedro Tamaroff informed us that in his work in progress, he proves that the associative operad is the only deformation of the Poisson operad; this offers additional supporting evidence for our result. A short summary of our result was presented at the 2019 Maple conference \cite{BD2019}.


\section{Preliminaries}

Throughout this paper we work over an arbitrary field $\k$ of characteristic $0$. 
The main question resolved in this paper concerns symmetric operads, so we focus our attention on them. The general definitions discussed in Section \ref{sec:inhomogen} are, however, equally valid for nonsymmetric or shuffle operads \cite{BD2016,LV2012}. The composite product of symmetric sequences (for which operads are monoids) is denoted $\circ$, and the unit for that product is denoted~$\mathbb{I}$.

We require all operads in this paper to be (nonnegatively) weight graded: this means that every component $\calP(n)$ admits a direct sum
decomposition $\calP(n)=\bigoplus_{k\ge 0}\calP(n)_{(k)}$ into components of weight $k$ for various $k\ge 0$ for which any operad composition of homogeneous elements of certain weights is a homogeneous element whose weight is the sum of the weights. In addition, we assume that all operads are reduced ($\calP(0)=0$) and connected ($\calP(1)_{(0)}=\k$ and $\calP(n)_{(0)}=0$ for $n>1$). A connected operad is automatically augmented, and we denote by $\overline{\calP}$ the augmentation ideal of $\calP$. Finally, we assume that all individual weight graded components $\calP(n)_{(k)}$ are finite-dimensional. 

Let us remark that each operad $\calP$ has one obvious weight grading where
 \[
\calP(n)_{(n-1)}=\calP(n) \text{ for } n \geq 1  \text{ and }
 \calP(n)_{(k)} =0 \text{ otherwise }.
 \]
For most commonly considered operads, those generated by binary operations subject to ternary relations, this grading is the most convenient one to use. However, beyond the binary generated operads, other weight gradings are occasionally more appropriate. The set-up we adopt throughout this article is that of \emph{standard} grading: we consider operads $\calP$ that are generated by elements of weight one. In this case every operad $\calP$ admits a \emph{standard presentation} $\calP\cong\calT(\calX)/(\calR)$, where $\calX:=\calP_{(1)}$ and $\calR$ is the minimal set of relations satisfied by $\calP_{(1)}$.

We shall frequently use the \emph{infinitesimal composition of collections}. For two collections $\calA$ and $\calB$, their infinitesimal composition $\calA\circ'\calB$ is the subcollection of the collection $\calA\circ(\calB\oplus\mathbb{I})$ spanned by tensors where elements of $\calB$ occur exactly once. It is defined for each of the three composite products of collections: nonsymmetric, shuffle, and symmetric. 


\section{Inhomogeneous distributive laws}\label{sec:inhomogen}

Suppose that $\calP$ and $\calQ$ are operads, and assume that $\calO$ is an operad which satisfies two properties:
\begin{itemize}
\item there is an injective map of operads $\jmath\colon\calQ\to \calO$,
\item the quotient operad $\calO/(\jmath(\overline{\calQ}))$ is isomorphic to $\calP$. 
\end{itemize}
Choose a splitting $\imath\colon\calP\to\calO$ on the level of weight graded symmetric sequences. We have a sequence of obvious maps of symmetric sequences
 \[
\calP\circ\calQ\hookrightarrow \calT(\calP\oplus\calQ)\to\calT(\imath(\calP)+\jmath(\calQ))\to\calT(\calO)\to\calO ;
 \]
Let us denote by $\eta\colon\calP\circ\calQ\to\calO$ the composite map. 

\begin{definition}
We say that the operad $\calO$ is obtained from $\calP$ and $\calQ$ by an \emph{inhomogeneous distributive rewriting rule} if the map $\eta$ is a surjection.
We say that the operad $\calO$ is obtained from $\calP$ and $\calQ$ by an \emph{inhomogeneous distributive law} if that map is an isomorphism.
\end{definition}

We remark that the notions of a distributive rewriting rule and a distributive law are well defined, i.e. they do not depend on the choice of the splitting $\imath$. Indeed, suppose that $\imath$ is a splitting for which the map $\eta_{\imath,\jmath}$ is surjective. Then every other splitting $\imath'$ differs from $\imath$ by some elements from $\imath(\calP)\circ\jmath(\calQ)$ that vanish in the quotient $\calO/(\jmath(\overline{\calQ}))$, and surjectivity is proved by easy induction on weight. Also, since we only work with operads with finite-dimensional weight graded components, surjectivity for $\eta'$ together with isomorphism for $\eta$ imply isomorphism for $\eta'$.

The way the notion of an inhomogeneous distributive law is defined makes sense conceptually but is difficult to check directly. Let us present an equivalent approach which is more user-friendly.  Let $\calP=\calT(\calX)/(\calR)$ and $\calQ=\calT(\calY)/(\calS)$ be two weight graded operads presented by generators and relations. An operad $\calO$ generated by $\calX\oplus\calY$ is obtained from $\calP$ and $\calQ$ by an \emph{inhomogeneous distributive rewriting rule} if the defining relations of $\calO$ are $\widetilde{\calR}\oplus\calD\oplus\calS$, with subcollections $\widetilde{\calR}$ and $\calD$ of the free operad $\calT(\calX\oplus\calY)$ satisfying two constraints. The first constraint postulates that there should exist a map of weight graded collections $\rho\colon \calR\to \calT(\calX)\circ\calT(\calY)\subset \calT(\calX\oplus\calY)$ such that the post-composition of $\rho$ with the projection $\calT(\calX\oplus\calY)\twoheadrightarrow \calT(\calX)$ is zero, and that the subcollection $\widetilde{\calR}$ consists of all elements of the form $r-\rho(r)$ with $r\in \calR$. The second constraint postulates that there should exist a map of weight graded collections $\lambda\colon \calY\circ'\calX\to \calT(\calX\oplus\calY)_{(2)}$ such that the postcomposition of $\lambda$ with the projection $\calT(\calX\oplus\calY)\twoheadrightarrow \calT(\calX)$ is zero and the subcollection $\calD$ consists of all elements $v-\lambda(v)$ with $v\in \calY\circ'\calX$. Note that this means that we have $\calO/(\calY)\cong\calP$; we choose some splitting $\alpha\colon\calP\to\calO$ on the level of weight graded collections, allowing us to define the maps
 \[
\calP \circ \calQ
\,\hookrightarrow\, 
\calT( \calP \oplus \calQ ) 
\,\rightarrow\,
\calT( \calO ) \to \calO.
 \]
Our definition of an inhomogeneous distributive rewriting rule ensures that the map $\eta$ obtained by composing these maps is a surjection on the level of underlying objects. An inhomogeneous distributive rewriting rule is said to be an \emph{inhomogeneous distributive law} if the map $\eta$ is an isomorphism on the level of underlying objects. 

For a particular case of quadratic operads, we recover precisely the definition of a filtered distributive law of \cite{DG2014}. The following result, proved analogously to \cite[Theorem 8.6.11]{LV2012}, provides a constructive approach to classification of inhomogeneous distributive laws. 

\begin{proposition}\label{prop:weight3}
Suppose that the operads $\calP$ and $\calQ$ are quadratic, and that the map $\eta$ is an isomorphism when restricted to elements of weight $3$. Then that map is an isomoprphism in all weights, so the operad $\calO$ is obtained from $\calP$ and $\calQ$ by an inhomogeneous distributive law.
\end{proposition}


\section{Main theorem}

\subsection{Parameters of the problem}

The operads we consider in this paper are $\Com$ and $\Lie$. The operad $\Com$ is generated by the symmetric sequence $\calX$ supported in arity $2$ whose component $\calX(2)$ is spanned by a commutative binary operation, which we denote by juxtaposition $a_1a_2$, and the relations $\calR$ are supported in arity $3$ and are given by the $\k S_3$-module generated by $(a_1 a_2) a_3 - (a_2 a_3) a_1$. The operad $\Lie$ is generated by the symmetric sequence $\calY$ supported in arity $2$ whose component $\calY(2)$ is spanned by an anti-commutative binary operation, which we denote by $[a_1,a_2]$, and the relations $\calS$ are supported in arity $3$ and are given by the $\mathbb{F}S_3$-module spanned by $[[a_1,a_2],a_3] - [[a_1,a_3],a_2] + [[a_2,a_3],a_1]$. 

Suppose that an operad $\calO$ is obtained from $\Com$ and $\Lie$ by an inhomogeneous distributive law. This operad is a quotient of $\calT(\calX\oplus\calY)$ by relations which, in this case, are all of weight $2$ and arity~$3$. The ternary component $\calT(\calX\oplus\calY)(3)$ has dimension $12$; for a basis we may choose the following operations:
\[
\begin{array}{l@{\quad}l@{\quad}l@{\quad}l@{\quad}l@{\quad}l}
(a_1a_2)a_3, &
(a_1a_3)a_2, &
(a_2a_3)a_1, &
[a_1a_2,a_3], &
[a_1a_3,a_2], &
[a_2a_3,a_1],
\\ {}
[a_1,a_2]a_3, &
[a_1,a_3]a_2, &
[a_2,a_3]a_1, &
[[a_1,a_2],a_3], &
[[a_1,a_3],a_2], &
[[a_2,a_3],a_1].
\end{array}
\]
From Section \ref{sec:inhomogen}, we know precisely the viable candidates for the relations of the operad~$\calO$; the general formulas simplify significantly because our operads have only binary operations and ternary relations: 
\begin{enumerate}[leftmargin=1.75em,label=(\alph*)]
\item The set of relations must contain $\calS$, which in our case means that we have to include 
\begin{equation}
\label{Jacobi}
[[a_1,a_2],a_3] - [[a_1,a_3],a_2] + [[a_2,a_3],a_1] .
\end{equation}
\item We should have a map of weight graded symmetric sequences
 \[
\lambda\colon \calY\circ'\calX\to \calT(\calX)\circ\calT(\calY)  
 \] 
constrained by requiring that under the canonical projection $\calT(\calX)\circ\calT(\calY)\twoheadrightarrow\calT(\calX)$ the image of $\lambda$ is sent to zero. The latter condition, together with the weight grading requirement, means that in fact we have a map 
 \[
\lambda\colon \calY\circ'\calX\to \calX\circ'\calY\oplus\calY\circ'\calY ,  
 \] 
and once we account for the symmetry of $[a_1a_2,a_3]$ under the transposition of $1$ and $2$, the relation we need, written using the basis elements, is of the form
\begin{equation}
\label{derivation}
[a_1a_2,a_3] - t_1 \big( \, [a_1,a_3]a_2+[a_2,a_3]a_1 \, \big) - t_2 \big( \, [[a_1,a_3],a_2]+[[a_2,a_3],a_1] \, \big) .
\end{equation}
\item We should have a map of weight graded symmetric sequences
 \[
\rho\colon \calR\to \calT(\calX)\circ\calT(\calY)  
 \]
constrained by requiring that under the canonical projection $\calT(\calX)\circ\calT(\calY)\twoheadrightarrow\calT(\calX)$ the image of $\rho$ is sent to zero. The latter condition, together with the weight grading requirement, means that in fact we have a map 
 \[
\rho\colon \calR\to \calX\circ'\calY\oplus\calY\circ'\calY . 
 \]
The $S_3$-module of relations $\calR$ is generated by the element $(a_1a_2)a_3-a_1(a_2a_3)$. This element satisfies two symmetry conditions. First, it is skew-symmetric under the transposition of $1$ and $3$, leading to a defining relation of the form
 \[
\left\{ \;
\begin{array}{r@{\,}l}
(a_1a_2)a_3 - (a_2a_3)a_1 
&
{}
- t_3 \big( \, [a_1,a_2]a_3+[a_2,a_3]a_1 \, \big) 
\\[1pt]
&
{}
- t_4 \big( \, [[a_1,a_2],a_3]+[[a_2,a_3],a_1] \, \big) 
- t_5 [[a_1,a_3],a_2] ,
\end{array}
\right.
 \]
and second, the sum over the cyclic permutations of arguments is zero, meaning that we must have $t_3=0$. Finally, one can use Relation \eqref{Jacobi} to simplify the Lie monomial part, obtaining a relation of the form 
\begin{equation}
\label{associative}
(a_1a_2)a_3 - (a_2a_3)a_1 - t_3 [[a_1,a_3],a_2] 
\end{equation}
with just one parameter.
\end{enumerate}

\begin{lemma}
\label{rowspace}
The submodule of $\calT(\calX\oplus\calY)(3)$ generated by relations \eqref{Jacobi}--\eqref{associative}
is the row space of the following $7 \times 12$ matrix $R$ with entries in the polynomial ring $\mathbb{F}[t_1,t_2,t_3]$.
The rows are the coefficient vectors of the spanning relations, and the $(i,j)$ entry is the coefficient in relation $i$ of basis
monomial $j$ in the ordered basis:
\begin{equation}
\label{matrix1}
\left[
\begin{array}{r@{\;\;}r@{\;\;}r@{\;\;}r@{\;\;}r@{\;\;}r@{\;\;}r@{\;\;}r@{\;\;}r@{\;\;}r@{\;\;}r@{\;\;}r}
 0  & 0  & 0  & 0  & 0  & 0  & 0  & 0  & 0  & 1  & -1  & 1\\
 0  & 0  & 0  & 1  & 0  & 0  & 0  & -t_1  & -t_1  & 0  & -t_2  & -t_2 \\
 0  & 0  & 0  & 0  & 1  & 0  & -t_1  & 0  & t_1  & -t_2  & 0  & t_2 \\ 
 0  & 0  & 0  & 0  & 0  & 1  & t_1  & t_1  & 0  & t_2  & t_2  & 0 \\ 
 1  & 0  & -1  & 0  & 0  & 0  & 0  & 0  & 0  & 0  & -t_3  & 0 \\
-1  & 1  & 0  & 0  & 0  & 0  & 0  & 0  & 0  & 0  & 0  & t_3 \\ 
\end{array}
\right]
\end{equation}
\end{lemma}

\begin{proof}
Relation \eqref{Jacobi} is skew-symmetric in all three arguments and so it is already a spanning set for the 
submodule it generates.
Relation \eqref{derivation} is symmetric in $a_1$ and $a_2$, and so it suffices to permute the arguments $a_1$, $a_2$, $a_3$ 
cyclically in order to obtain a spanning set for the submodule that it generates.
Relation \eqref{associative} generates a two-dimensional submodule, so it is enough to take one of its two cyclic permutations.
\end{proof}

\begin{proposition}
The reduced row echelon form of the matrix \eqref{matrix1} is
\begin{equation}
\label{4parameters}
\left[
\begin{array}{r@{\quad}r@{\quad}r@{\quad}r@{\quad}r@{\quad}r@{\quad}r@{\quad}r@{\quad}r@{\quad}r@{\quad}r@{\quad}r}
1 & 0 & -1 & 0 & 0 & 0 & 0 & 0 & 0 & 0 & -t_3 & 0
\\
0 & 1 & -1 & 0 & 0 & 0 & 0 & 0 & 0 & 0 & -t_3 & t_3
\\
0 & 0 & 0 & 1 & 0 & 0 & 0 & -t_1 & -t_1 & 0 & -t_2 & -t_2
\\
0 & 0 & 0 & 0 & 1 & 0 & -t_1 & 0 & t_1 & 0 & -t_2 & 2 t_2
\\
0 & 0 & 0 & 0 & 0 & 1 & t_1 & t_1 & 0 & 0 & 2 t_2 & -t_2
\\
0 & 0 & 0 & 0 & 0 & 0 & 0 & 0 & 0 & 1 & -1 & 1
\end{array}
\right]
\end{equation}
\end{proposition}

\subsection{Cubic consequences of the quadratic relations}

We now consider the arity $4$ component of the relations determined by the row space of a matrix of the form 
\eqref{4parameters}. The component $\calT(\calX\oplus\calY)(4)$ has dimension 120, 
and an ordered monomial basis which may be described as follows.
For uniformity of notation, we temporarily write $a \circ b = ab$ and $a \bullet b = [a,b]$ for the commutative and
anticommutative operations respectively.
There are 96 monomials of association type 1, namely $( ( ( x_1 \ast_1 x_2 ) \ast_2 x_3 ) \ast_3 x_4).\sigma$ where $\sigma(1)<\sigma(2)$ and 
$\ast_1, \ast_2, \ast_3 \in \{ \circ, \bullet \}$; the order is lexicographic by the permutation $\sigma$, and then by the triple of operation symbols with $\circ \prec \bullet$.
There are 24 monomials of association type 2, namely $(( x_1 \ast_1 x_2 ) \ast_2 ( x_3 \ast_3 x_4 )).\sigma$ where $\sigma(1)<\sigma(2)$, $\sigma(3)<\sigma(4)$, and $\sigma(1) \prec \sigma(3)$; the order is lexicographic as before.

\begin{example}
The last eight monomials in association types 1 and 2 are as follows:
\[
\begin{array}{l@{\quad}l@{\quad}l@{\quad}l}
((x_3 x_4)x_2)x_1{,} &
([x_3,x_4]x_2)x_1{,} &
[x_3x_4{,}x_2]x_1{,} &
[[x_3,x_4],x_2]x_1{,} 
\\[1mm]
[(x_3x_4)x_2{,}x_1]{,} &
[[x_3,x_4]x_2{,}x_1]{,} &
[[x_3x_4{,}x_2]{,}x_1]{,} &
[[[x_3,x_4],x_2],x_1];
\\[1mm]
(x_1x_4)(x_2x_3){,} &
(x_1x_4)[x_2,x_3]{,} &
[x_1,x_4](x_2x_3){,} &
[x_1,x_4],[x_2,x_3]{,} 
\\[1mm]
[x_1x_4{,}x_2x_3]{,} &
[x_1x_4{,}[x_2,x_3]]{,} &
[[x_1,x_4]{,}x_2x_3]{,} &
[[x_1,x_4],[x_2,x_3]].
\end{array}
\]
\end{example}

To determine the arity $4$ component of the quotient of the free operad $\calT(\calX\oplus\calY)$ by the relations $\{r_i\}_{i=1,\ldots,6}$ we are considering, we should compute the quotient of $\calT(\calX\oplus\calY)(4)$ by the $S_4$-submodule generated by all partial compositions 
$r_i \circ_j \mu$,
$r_i \circ_j \lambda$,
$\mu \circ_k r_i$,
$\lambda \circ_k r_i$,
where $\mu, \lambda$ represent the commutative and anticommutative operations respectively, and the limits on the indices are
$1 \le i \le 6$, $1 \le j \le 3$, and $1 \le k \le 2$.
Since $\mu$ and $\lambda$ generate one-dimensional $S_2$-modules, we need to include only the cases $k=1$, namely  
$\mu \circ_1 r_i$,
$\lambda \circ_1 r_i$.
Altogether we obtain a set of 48 partial compositions that are consequences of our relations, and after applying all 24 permutations of the arguments, we have a spanning set of 1152 elements for the space of all arity 4 consequences of our relations. Note that the monomials of arity 4 which occur in the elements we constructed may require \emph{straightening} in order to belong to the set of basis monomials above. In other words, in some cases we must replace expressions of the form $f \circ g$ and $f \bullet g$ by $g \circ f$ and $-(g \bullet f)$ respectively.

\begin{example}
Row 1 of matrix \eqref{4parameters} represents the relation
\[
r_1(x_1,x_2,x_3) = 
(x_1x_2)x_3
- (x_2x_3)x_1
- t_3 [[x_1,x_3],x_2].
\]
The partial composition $r_1\circ_3\lambda$ amounts to replacing $x_3$ by $[x_3,x_4]$:
\[
r_1 \circ_3 \lambda = (x_1x_2)[x_3,x_4]
- (x_2[x_3,x_4])x_1
- t_3 [[x_1,[x_3,x_4]],x_2].
\]
Straightening the monomials in this relation expresses $r_1 \circ_3 \lambda$ as a linear combination of the ordered monomial basis:
\[
r_1 \circ_3 \lambda  = 
(x_1x_2)[x_3,x_4]
- ([x_3,x_4]x_2)x_1
+ t_3 [[[x_3,x_4],x_1],x_2].
\]
\end{example}

\subsection{Statement and proof of the classification theorem}

Our main result is the following classification theorem mentioned in the introduction. 

\begin{theorem}
The only operads obtained from the symmetric operads $\Com$ and $\Lie$ by an inhomogeneous distributive law 
are defined by the following relations:
\begin{align}
\label{relations1}
&
\left\{ \quad
\begin{array}{l}
(x_1x_2)x_3 - x_1(x_2x_3) = 0,
\\ {}
[x_1x_2,x_3] = 0,
\\ {}
[[x_1,x_2],x_3] + [[x_2,x_3],x_1] + [[x_3,x_1],x_2] = 0.
\end{array}
\right.
\\[1mm]
\label{relations2}
&
\left\{ \quad
\begin{array}{l}
(x_1x_2)x_3 - x_1(x_2x_3) - q [[x_1,x_3],x_2] = 0 \;\; (q \in \k),
\\ {}
[x_1x_2,x_3] - [x_1,x_3]x_2 - x_1[x_2,x_3] = 0,
\\ {}
[[x_1,x_2],x_3] + [[x_2,x_3],x_1] + [[x_3,x_1],x_2] = 0.
\end{array}
\right.
\end{align}
To classify such operads up to isomorphism, one has to replace ``$q \in \k$'' by ``$q \in \k/(\k^\times)^2$''.
\end{theorem}

\begin{proof}
By Proposition \ref{prop:weight3}, it is enough to check that the map $\eta$ is an isomorphism when restricted to cubic elements, which in our case means elements of arity $4$, since our operads are generated by binary operations.

The 1152 elements spanning the space of all arity 4 consequences of our relations give us a $1152 \times 120$-matrix $M$ which has entries in the polynomial ring $\k[t_1,t_2,t_3]$, which we equip with the \textsl{deglex} (\textsl{tdeg} in \texttt{Maple}) monomial order
$t_1 \succ t_2 \succ t_3$.

Since $\k[t_1,t_2,t_3]$ is not a PID, the matrix $M$ has no Smith form, but since many entries of $M$ are $\pm 1$, we can compute a partial Smith form; see \cite[Chapter 8]{BD2016} and \cite{BD2017}. The result is a block matrix 
 \[
\begin{pmatrix}
I_{96} & 0_{96\times 24}\\
0_{1056\times 96} & L'
\end{pmatrix} , 
 \]
where the lower right block $L'$ of size $1056 \times 24$ has many zero rows. (This calculation took less than 25 seconds using Maple 18 on a MacBook Pro.) Deleting the zero rows, we obtain a matrix $L$ of size $372 \times 24$ which contains $126$ distinct elements of $\k[t_1,t_2,t_3]$.
We replace each of these elements by its monic form and obtain a set $S$ of $56$ distinct elements of degrees $2$ and~$3$. Finally, we compute the reduced Gr\"obner basis for the ideal $I$ generated by $S$ and obtain the set 
 \[
t_2, \quad t_3(t_1-1), \quad t_1(t_1-1).
 \]
The zero set of these consists of the point $(0,0,0)$ and the line $(1,0,t_3)$. This proves the first statement of the theorem.

To establish the classification up to isomorphism, we note that a hypothetical isomorphism between two such operads $\calO$ and $\calO'$ must send the symmetric generator of $\calO$ into a nonzero scalar multiple of the symmetric generator of $\calO'$, and the anti-symmetric generator of $\calO$ into a nonzero scalar multiple of the anti-symmetric generator of $\calO'$. The simultaneous rescaling by the same factor does not change the relations, so we may assume that the first scalar is equal to one. The classification result follows, since such an isomorphism multiplies $q$ by a nonzero square.
\end{proof}

We remark that the operad defined by~\eqref{relations1} corresponds to the trivial distributive law between $\Com$ and $\Lie$ \cite[Sec.~8.6.4]{LV2012}, and the operad defined by~\eqref{relations2} is the Livernet--Loday deformation of $\Poisson$ into $\Ass$ over a field of characteristic zero~\cite{MR2006}. In fact, over any field $\k$ of characteristic different from two, the formula $\phi(x_1\star x_2):=x_1x_2+[x_1,x_2]$ defines an isomorphism from the operad $\Ass$ to the operad defined by~\eqref{relations2} for $q=1$. Over a field $\k$ of characteristic two, an operad from that family cannot be isomorphic to the associative operad, as already the space of binary operations is a different $\k S_2$-module. 

As a concluding remark, we would like to note that even the existence of one non-trivial inhomogeneous distributive law should be regarded as a miracle of a sort. For instance, if we replace the operad of Lie algebras by the operad $\NLie_2$ of two-step nilpotent Lie algebras with the defining relations $[[x_1,x_2],x_3]=0$, a computation similar to ours shows that the only inhomogeneous distributive law between that operad and the operad $\Com$ is the trivial one. In particular, the Leibniz rule does not give a distributive law, as was noted by the second author many years ago, see \cite[Exercise~8.10.12]{LV2012}.



\begin{thebibliography}{99}

\bibitem{Beck1969}
\textsc{J. Beck}:
Distributive laws.
\emph{Seminar on Triples and Categorical Homology Theory (ETH, Z\"urich, 1966/67)},
119--140.
Editors: B. Eckmann, M. Tierney.
Lecture Notes in Mathematics, 80.
Springer-Verlag, Berlin-Heidelberg, 1969.
Republished in:
\emph{Reprints in Theory and Applications of Categories}, 18 (2008) 1--303.
\url{www.tac.mta.ca/tac/reprints/articles/18/tr18.pdf}

\bibitem{BD2016}
\textsc{M. Bremner, V. Dotsenko}:
\emph{Algebraic Operads: An Algorithmic Companion}.
Chapman and Hall/CRC, Boca Raton, FL, USA, 2016.

\bibitem{BD2017}
\textsc{M. Bremner, V. Dotsenko}:
Classification of regular parametrised one-relation operads.
\emph{Canadian Journal of Mathematics}, 69 (2017), no. 5, 992--1035.

\bibitem{BD2019}
\textsc{M. Bremner, V. Dotsenko}:
Distributive Laws Between the Operads $\Lie$ and $\Com$: Extended Abstract.
To appear in 
\emph{Proceedings of the Maple Conference 2019, Waterloo, Ontario, Canada, October 15--17, 2019}.
Communications in Computer and Information Science, Springer.

\bibitem{BM2018}
\textsc{M. Bremner, M. Markl}:
Distributive laws between the Three Graces. 
To appear in \emph{Theory and Applications of Categories}.
See also the arXiv preprint: \url{arXiv:1809.08191}.

\bibitem{DG2014}
\textsc{V. Dotsenko, J. Griffin}:
Cacti and filtered distributive laws.
\emph{Algebraic and Geometric Topology} 
14 (2014), no. 6, 3185--3225. 

\bibitem{DT2018}
\textsc{V. Dotsenko and P. Tamaroff}:
Endofunctors and Poincar\'e--Birkhoff--Witt theorems. 
\emph{ArXiv preprint} \texttt{arXiv:1804.06485}.

\bibitem{LV2012}
\textsc{J.-L. Loday, B. Vallette}:
\emph{Algebraic Operads}.
Grundlehren der mathematischen Wissenschaften, 346.
Springer, Heidelberg, Germany, 2012.

\bibitem{Markl1996}
\textsc{M. Markl}:
Distributive laws and Koszulness. 
\emph{Annales de l'Institut Fourier (Grenoble)} 
46 (1996), no. 2, 307--323. 

\bibitem{MR2006}
\textsc{M. Markl, E. Remm}:
Algebras with one operation including Poisson and other Lie-admissible algebras. 
\emph{Journal of Algebra} 
299 (2006), no. 1, 171--189. 

\bibitem{PP} 
\textsc{A. Polishchuk, L. Positselski}: 
Quadratic algebras. University Lecture Series, 37. American Mathematical Society,
 Providence, RI, 2005. 


\end{thebibliography}
\end{document}